\documentclass[12pt]{article}
\usepackage{amsmath,amssymb,amsthm}
\numberwithin{equation}{section}
\usepackage{hyperref}
\usepackage{units}
\usepackage{color}
\usepackage[T1]{fontenc}
\usepackage[utf8]{inputenc}
\usepackage{authblk}
\usepackage{bm}
\usepackage{graphicx}

\usepackage{datetime}
\usepackage{color}

\usepackage[toc,page]{appendix}

\newtheorem{thm}{Theorem}[section]
\newtheorem{theorem}[thm]{Theorem}
\newtheorem{lemma}[thm]{Lemma}

\title{Symmetry properties of finite sums involving generalized Fibonacci numbers}

\author[]{Kunle Adegoke \thanks{adegoke00@gmail.com}}
\author[]{Oluwaseyi Oshin \thanks{oluwaseyioshin@gmail.com}}

\affil{Department of Physics and Engineering Physics, \mbox{Obafemi Awolowo University}, 220005 Ile-Ife, Nigeria}

\begin{document}

\date{}

\maketitle

\begin{abstract}
\noindent We extend a result of I.~J.~Good and prove more symmetry \mbox{properties} of sums involving generalized Fibonacci numbers.
\end{abstract}

\section{Introduction}
The generalized Fibonacci numbers $G_i$, $i \ge 0$, with which we are mainly concerned in this paper, are defined through the second order recurrence relation $G_{i+1} = G_i + G_{i-1}$, where the seeds $G_0$ and $G_1$ need to be specified. As particular cases, when $G_0 = 0$ and $G_1= 1$, we have the Fibonacci numbers, denoted $F_i$, while when $G_0= 2$ and $G_1= 1$, we have the Lucas numbers, $L_i$.

\bigskip

I.~J.~Good~\cite{good} proved the symmetry property:
\begin{equation}\label{equ.good}
F_{q} \sum_{k = 1}^n {\frac{{( - 1)^{k} }}{{G_{k} G_{k + q} }}}  = F_{n} \sum_{k = 1}^q {\frac{{( - 1)^{k} }}{{G_{k} G_{k + n} }}}\,,
\end{equation}
where $q$ and $n$ are nonnegative integers, and all the  numbers $G_1$, $G_2$, \ldots, $G_{n+q}$ are nonzero.

\bigskip

The identity~\eqref{equ.good} is a particular case (corresponding to setting $p=1$) of the following result, to be proved in this present paper:
\begin{equation}\label{equ.sc8meci}
F_{pq} \sum_{k = 1}^n {\frac{{( - 1)^{pk} }}{{G_{pk} G_{pk + pq} }}}  = F_{pn} \sum_{k = 1}^q {\frac{{( - 1)^{pk} }}{{G_{pk} G_{pk + pn} }}}\,,
\end{equation} 
where $q$, $p$ and $n$ are nonnegative integers, and all the numbers $G_p$, $G_{2p}$, \ldots, $G_{pn+pq}$ are nonzero.

\bigskip

In the limit as $n$ approaches infinity, and specializing to Fibonacci numbers, the identity~\eqref{equ.sc8meci} gives
\begin{equation}\label{equ.zbajbkj}
\begin{split}
\sum\limits_{k = 1}^\infty  {\frac{{( - 1)^{pk} }}{{F_{pk} F_{pk + pq} }}}  &= \frac{1}{{F_{pq} }}\sum\limits_{k = 1}^q {\left\{ {\frac{{( - 1)^{pk} }}{{F_{pk} }}\mathop {\lim }\limits_{n \to \infty } \left( {\frac{{F_{pn} }}{{F_{pk + pn} }}} \right)} \right\}}\\
&= \frac{1}{{F_{pq} }}\sum\limits_{k = 1}^q {\frac{{( - 1)^{pk} }}{{\phi ^{pk} F_{pk} }}}\,,
\end{split}
\end{equation}
where $\phi=(1+\sqrt 5)/2$ is the golden ratio.

\bigskip

The identity~\eqref{equ.zbajbkj} generalizes Bruckman and Good's result (identity~(19) of~\cite{bruckman}, which corresponds to setting $q=1$ in~\eqref{equ.zbajbkj}).

\bigskip

In sections~\ref{sec.main1} --~\ref{sec.more} we will prove identity~\eqref{equ.sc8meci} and discover more symmetry properties of sums involving generalized Fibonacci numbers. In section~\ref{sec.horadam} we shall extend the discussion to Horadam sequences $W_i$ and $U_i$ by proving
\begin{equation}
U_{pq} \sum_{k = 1}^n {\frac{{Q^{pk} }}{{W_{pk} W_{pk + pq} }}}  = U_{pn} \sum_{k = 1}^q {\frac{{Q^{pk} }}{{W_{pk} W_{pk + pn} }}}
\end{equation}
and
\begin{equation}
U_{2pq} \sum_{k = 1}^{2n} {\frac{{(\pm Q^p)^{k} }}{{W_{pk} W_{pk + 2pq} }}}  = U_{2pn} \sum_{k = 1}^{2q} {\frac{{(\pm Q^p)^{k} }}{{W_{pk} W_{pk + 2pn} }}}\,,
\end{equation}
for integers $p$, $q$, $Q$ and $n$, thereby extending Andr\'e-Jeannin's result (Theorem~1 of~\cite{jeannin}) and further generalizing the identity~\eqref{equ.sc8meci}.
\section{Required identities}
\subsection{Telescoping summation identities}
The following telescoping summation identities are special cases of the more general identities proved in~\cite{adegoke}.
\begin{lemma}\label{FS}
If $f(k)$ is a real sequence and $u$, $v$ and $w$ are positive integers, then
\[
\sum_{k = 1}^w {\left[ {f(uk + uv) - f(uk)} \right]}  = \sum_{k = 1}^v {\left[f(uk + uw)-f(uk)\right]}\,. 
\]
\end{lemma}

\begin{lemma}\label{FS1}
If $f(k)$ is a real sequence and $u$, $v$ and $w$ are positive integers such that $v$ is even and $w$ is even, then
\[
\sum_{k = 1}^w {( \pm 1)^{k - 1} \left( {f(uk + uv) - f(uk)} \right)}  = \sum_{k = 1}^v {( \pm 1)^{k - 1} \left( {f(uk + uw) - f(uk)} \right)}\,. 
\]
	 
\end{lemma}

\begin{lemma}\label{FS2}
If $f(k)$ is a real sequence and $u$, $v$ and $w$ are positive integers such that $vw$ is odd, then
\[
\sum_{k = 1}^w {( - 1)^{k - 1} \left( {f(uk + uv) + f(uk)} \right)}  = \sum_{k = 1}^v {( - 1)^{k - 1} \left( {f(uk + uw) + f(uk)} \right)}\,. 
\]	
\end{lemma}

\subsection{Product of a Fibonacci number and a generalized Fibonacci number}
\begin{lemma}[Howard~\cite{howard}, Corollary 3.5]\label{thm.fg}
For integers $a$, $b$, $c$,
\[
F_aG_{2b+a+c}=
\begin{cases}
F_{a+b}G_{a+b+c}-F_bG_{b+c} & \text{if $a$ is even,}\\
F_{a+b}G_{a+b+c}+F_bG_{b+c} & \text{if $a$ is odd.}
\end{cases}
\]
\end{lemma}

\subsection{Product of a Lucas number and a generalized Fibonacci number}
\begin{lemma}[Vajda~\cite{vajda}, Formula 10a]\label{thm.lg}
For integers $a$, $b$,
\[
L_aG_b=
\begin{cases}
G_{b+a}+G_{b-a} & \text{if $a$ is even,}\\
G_{b+a}-G_{b-a} & \text{if $a$ is odd.}
\end{cases}
\]
\end{lemma}
\subsection{Difference of products of a Fibonacci number and a generalized Fibonacci number}
\begin{lemma}[Vajda~\cite{vajda}, Formula 21]\label{thm.diff}
For integers $a$, $b$,
\[
F_bG_a-F_aG_b=(-1)^aG_0F_{b-a}\,.
\]
\end{lemma}

\section{Main Results: Symmetry properties}\label{sec.main}
\subsection{Sums of products of reciprocals}\label{sec.main1}
\begin{theorem}\label{thm.moyg61s}
If $n$ and $q$ are nonnegative integers and $p$ is a nonzero integer, then
\[
F_{pq} \sum_{k = 1}^n {\frac{{( - 1)^{pk} }}{{G_{pk} G_{pk + pq} }}}  = F_{pn} \sum_{k = 1}^q {\frac{{( - 1)^{pk} }}{{G_{pk} G_{pk + pn} }}}\,.
\]

\end{theorem}
\begin{proof}
Dividing through the identity in Lemma~\ref{thm.diff} by $G_aG_b$ and setting $b=pk+pq$ and $a=pk$, we have:
\begin{equation}\label{equ.azh431k}
\frac{F_{pk+pq}}{G_{pk+pq}}-\frac{F_{pk}}{G_{pk}}=(-1)^{pk}\frac{G_0F_{pq}}{G_{pk}G_{pk+pq}}\,.
\end{equation}
Similarly,
\begin{equation}\label{equ.gw81jiy}
\frac{F_{pk+pn}}{G_{pk+pn}}-\frac{F_{pk}}{G_{pk}}=(-1)^{pk}\frac{G_0F_{pn}}{G_{pk}G_{pk+pn}}\,.
\end{equation}
We now use the sequence $f(k)=F_k/G_k$ in Lemma~\ref{FS} with $u=p$, $v=q$ and $w=n$, while taking into consideration identities~\eqref{equ.azh431k} and~\eqref{equ.gw81jiy}.
\end{proof}
\begin{theorem}\label{thm.h4ozka6}
If $n$ and $q$ are nonnegative \underline{even} integers and $p$ is a nonzero integer, then
\[
F_{pq} \sum_{k = 1}^{n} {\frac{{( \pm 1)^{k(p - 1)} }}{{G_{pk} G_{pk + pq} }}}  = F_{pn} \sum_{k = 1}^{q} {\frac{{( \pm 1)^{k(p - 1)} }}{{G_{pk} G_{pk + pn} }}}\,.
\]

\end{theorem}
\begin{proof}
We use the sequence $f(k)=F_k/G_k$ in Lemma~\ref{FS1} with $u=p$, $v=q$ and $w=n$.
\end{proof}
\subsection{First-power sums}\label{sec.main2}
\begin{theorem}\label{T1}
	If $p$, $q$, $n$ and $t$ are integers such that $pqn$ is odd, then
	\begin{equation}\label{equ.r6v9g43}
		L_{pq} \sum_{k = 1}^{2n} {( \pm 1)^{k - 1} G_{pk + pq + t} }  = L_{pn} \sum_{k = 1}^{2q} {( \pm 1)^{k - 1} G_{pk + pn + t} }\,, 
	\end{equation}
	\begin{equation}\label{equ.r0qdvj9}
		L_{pq} \sum_{k = 1}^{n} { G_{2pk + pq + t} }  = L_{pn} \sum_{k = 1}^{q} {G_{2pk + pn + t} }\,. 
	\end{equation}

\end{theorem}
\begin{proof}
Consider the generalized Fibonacci sequence $f(k)=G_{k+t}$. If we choose $u=p$, $v=2q$ and $w=2n$, then Lemma~\ref{FS1} gives
\begin{equation}\label{equ.cuj0nfu}
\sum_{k = 1}^{2n} {(\pm 1)^{k-1}\left( {G_{pk + 2pq + t}-G_{pk + t} } \right)}  = \sum_{k = 1}^{2q} {(\pm 1)^{k-1}\left( {G_{pk + 2pn + t}-G_{pk + t} } \right)}\,.
\end{equation}
But from the second identity of Lemma~\ref{thm.lg} we have
\begin{equation}\label{equ.cm8jui8}
G_{pk + 2pq + t}  - G_{pk + t}  = L_{pq} G_{pk + pq + t},\quad\mbox{$pq$ odd}\,, 
\end{equation}
and
\begin{equation}\label{equ.c37r3b0}
G_{pk + 2pn + t}  - G_{pk + t}  = L_{pn} G_{pk + pn + t},\quad\mbox{$pn$ odd}\,.
\end{equation}
Using \eqref{equ.cm8jui8} and~\eqref{equ.c37r3b0} in~\eqref{equ.cuj0nfu}, identity~\eqref{equ.r6v9g43} is proved. 

\bigskip

The proof of identity~\eqref{equ.r0qdvj9} is similar, we use the sequence $f(k)=G_{2k+t}$ in Lemma~\ref{FS} with \mbox{$u=2p$}, $v=q$ and $w=n$.

\end{proof}
\begin{theorem}\label{TX}
	If $p$, $q$, $n$ and $t$ are integers such that $pqn$ is odd or $q$ and $n$ are even, then
	\[
		F_{pq} \sum_{k = 1}^{n} {( - 1)^{k - 1} G_{2pk + pq + t} }  = F_{pn} \sum_{k = 1}^{q} {( - 1)^{k - 1} G_{2pk + pn + t} } \,.
	\]
\begin{proof}
Consider the sequence $f(k)=F_kG_{k+t}$. If we choose \mbox{$u=p$}, $v=q$ and $w=n$, then Lemma~\ref{FS2} gives
\begin{equation}\label{equ.isldtcl}
\begin{split}
&\sum_{k = 1}^{n} {(-1)^{k-1}\left( {F_{pk+pq}G_{pk + pq + t}+F_{pk}G_{pk + t} } \right)}\\
&\quad  = \sum_{k = 1}^{q} {(-1)^{k-1}\left( {F_{pk+pn}G_{pk + pn + t}+F_{pk}G_{pk + t} } \right)}\,.
\end{split}
\end{equation}
From the second identity of Lemma~\ref{thm.fg} we have
\begin{equation}\label{equ.nzzx5ve}
F_{pk+pq}G_{pk + pq + t}+F_{pk}G_{pk + t}  = F_{pq} G_{2pk + pq + t},\quad\mbox{$pq$ odd}\,, 
\end{equation}
and
\begin{equation}\label{equ.gn8wgtn}
F_{pk+pn}G_{pk + pn + t}+F_{pk}G_{pk + t}  = F_{pn} G_{2pk + pn + t},\quad\mbox{$pn$ odd}\,.
\end{equation}
The theorem then follows from using \eqref{equ.nzzx5ve} and~\eqref{equ.gn8wgtn} in~\eqref{equ.isldtcl}. If $q$ and $n$ are even then we use  $f(k)=F_kG_{k+t}$ with \mbox{$u=p$}, $v=q$ and $w=n$ in Lemma~\ref{FS1} together with the first identity of Lemma~\ref{thm.fg}.
\end{proof}

\end{theorem}
\begin{theorem}\label{T5}
	If $p$, $q$, $n$ and $t$ are integers such that $p$ is even or $q$ and $n$ are even, then
	\[
	F_{pq} \sum_{k = 1}^n {G_{2pk + pq + t} }  = F_{pn} \sum_{k = 1}^q {G_{2pk + pn + t} }\,.
	\]
	
\end{theorem}
\begin{proof}
Consider the sequence $f(k)=F_kG_{k+t}$. Lemma~\ref{FS} with  \mbox{$u=p$}, $v=q$ and $w=n$ gives
\begin{equation}\label{equ.tehey2a}
\begin{split}
&\sum_{k = 1}^{n} {\left( {F_{pk+pq}G_{pk + pq + t}-F_{pk}G_{pk + t} } \right)}\\
&\quad  = \sum_{k = 1}^{q} {\left( {F_{pk+pn}G_{pk + pn + t}-F_{pk}G_{pk + t} } \right)}\,.
\end{split}
\end{equation}
From the first identity of Lemma~\ref{thm.fg} we have
\begin{equation}\label{equ.vtbejuo}
F_{pk+pq}G_{pk + pq + t}-F_{pk}G_{pk + t}  = F_{pq} G_{2pk + pq + t},\quad\mbox{$pq$ even}\,, 
\end{equation}
and
\begin{equation}\label{equ.p2n0qq9}
F_{pk+pn}G_{pk + pn + t}-F_{pk}G_{pk + t}  = F_{pn} G_{2pk + pn + t},\quad\mbox{$pn$ even}\,.
\end{equation}
Using \eqref{equ.vtbejuo} and~\eqref{equ.p2n0qq9} in~\eqref{equ.tehey2a}, Theorem~\ref{T5} is proved.
\end{proof}
\begin{theorem}\label{T9}
	If $p$, $q$, $n$ and $t$ are integers such that $p$ is even, then
	\[
	F_{pq} \sum_{k = 1}^{2n} {( \pm 1)^{k - 1} G_{pk + pq+t} }  = F_{pn} \sum_{k = 1}^{2q} {( \pm 1)^{k - 1} G_{pk + pn+t} }\,.
	\]
		
\end{theorem}
\begin{proof}
Consider the sequence $f(k)=F_kG_{k+t}$. Lemma~\ref{FS1} with  \mbox{$u=p$}, $v=2q$ and $w=2n$ gives
\begin{equation}\label{equ.q4swp2w}
\begin{split}
&\sum_{k = 1}^{2n} {(\pm 1)^{k-1}\left( {F_{pk+2pq}G_{pk + 2pq + t}-F_{pk}G_{pk + t} } \right)}\\
&\quad  = \sum_{k = 1}^{2q} {(\pm 1)^{k-1}\left( {F_{pk+2pn}G_{pk + 2pn + t}-F_{pk}G_{pk + t} } \right)}\,.
\end{split}
\end{equation}
From identities~\eqref{equ.vtbejuo} and~\eqref{equ.p2n0qq9} we have
\begin{equation}\label{equ.sqqfbax}
F_{pk+2pq}G_{pk + 2pq + t}-F_{pk}G_{pk + t}  = F_{2pq} G_{2pk + 2pq + t}\,, 
\end{equation}
and
\begin{equation}\label{equ.i6ogk6d}
F_{pk+2pn}G_{pk + 2pn + t}-F_{pk}G_{pk + t}  = F_{2pn} G_{2pk + 2pn + t}\,.
\end{equation}
Using \eqref{equ.sqqfbax} and~\eqref{equ.i6ogk6d} in~\eqref{equ.q4swp2w}, Theorem~\ref{T9} is proved.
\end{proof}
\begin{theorem}\label{T8}
	If $p$, $q$, $n$ and $t$ are integers such that $p$ is even and $nq$ is odd, then
	\[
	L_{pq} \sum_{k = 1}^{n} {( - 1)^{k - 1} G_{2pk + pq + t} }  = L_{pn} \sum_{k = 1}^{q} {( - 1)^{k - 1} G_{2pk + pn + t} }\,,
	\]

\end{theorem}
\begin{proof}
Consider the sequence $f(k)=G_{2k+t}$. If we choose \mbox{$u=2p$}, $v=q$ and $w=n$, then Lemma~\ref{FS2} gives
\begin{equation}\label{equ.qyi2urw}
\begin{split}
&\sum_{k = 1}^{n} {(-1)^{k-1}\left( {G_{2pk + 2pq + t}+G_{2pk + t} } \right)}\\  
&\quad= \sum_{k = 1}^{q} {(-1)^{k-1}\left( {G_{2pk + 2pn + t}+G_{2pk + t} } \right)},\quad\mbox{$nq$ odd}\,.
\end{split}
\end{equation}
From the first identity in Lemma~\ref{thm.lg}, we have
\begin{equation}\label{equ.jro3cuv}
G_{2pk + 2pq + t}  + G_{2pk + t}  = L_{pq} G_{2pk + pq + t},\quad\mbox{$pq$ even}\,, 
\end{equation}
and
\begin{equation}\label{equ.i2dvppw}
G_{2pk + 2pn + t}  + G_{2pk + t}  = L_{pn} G_{2pk + pn + t},\quad\mbox{$pn$ even}\,.
\end{equation}
Using \eqref{equ.jro3cuv} and~\eqref{equ.i2dvppw} in~\eqref{equ.qyi2urw}, Theorem~\ref{T8} is proved. 
\end{proof}
\subsection{More sums involving products of reciprocals}\label{sec.more}
\begin{theorem}
If $p$, $q$, $n$ and $t$ are positive integers such that $pnq$ is odd, then
\begin{equation}\label{equ.ur0b9zu}
L_{pq} \sum_{k = 1}^{2n} {\frac{{( \pm 1)^{k - 1} G_{pk + pq + t} }}{{G_{pk + t} G_{pk + 2pq + t} }}}  = L_{pn} \sum_{k = 1}^{2q} {\frac{{( \pm 1)^{k - 1} G_{pk + pn + t} }}{{G_{pk + t} G_{pk + 2pn + t} }}}\,, 
\end{equation}
\begin{equation}\label{equ.bd2tciu}
L_{pq} \sum_{k = 1}^{n} {\frac{{ G_{2pk + pq + t} }}{{G_{2pk + t} G_{2pk + 2pq + t} }}}  = L_{pn} \sum_{k = 1}^{q} {\frac{{ G_{2pk + pn + t} }}{{G_{2pk + t} G_{2pk + 2pn + t} }}}\,.
\end{equation}
\end{theorem}
\begin{proof}
Use of $f(k)=1/G_{k+t}$ in Lemma~\ref{FS1} with $u=p$, $v=2q$ and $w=2n$, noting the identites~\eqref{equ.cm8jui8} and~\eqref{equ.c37r3b0} proves identity~\eqref{equ.ur0b9zu}. To prove identity~\eqref{equ.bd2tciu}, we use $f(k)=1/G_{2k+t}$ in Lemma~\ref{FS} with $u=p$, $v=q$ and $w=n$, together with the second identity in Lemma~\ref{thm.lg}.
\end{proof}
\begin{theorem}
If $p$, $q$, $n$ and $t$ are positive integers such that $p$ is even and $nq$ is odd, then
\[
L_{pq} \sum_{k = 1}^{n} {\frac{{( - 1)^{k - 1} G_{2pk + pq + t} }}{{G_{2pk + t} G_{2pk + 2pq + t} }}}  = L_{pn} \sum_{k = 1}^{q} {\frac{{( - 1)^{k - 1} G_{2pk + pn + t} }}{{G_{2pk + t} G_{2pk + 2pn + t} }}}\,. 
\]
\end{theorem}
\begin{proof}
Use $f(k)=1/G_{2k+t}$ in Lemma~\ref{FS2} with $u=p$, $v=q$ and $w=n$, employing the identities~\eqref{equ.jro3cuv} and~\eqref{equ.i2dvppw}.
\end{proof}
\begin{theorem}
If $p$, $q$, $n$ and $t$ are positive integers such that $p$ is even or $n$ and $q$ are even, then
\[
F_{pq} \sum_{k = 1}^n {\frac{{G_{2pk + pq + t} }}{{F_{pk} G_{pk + t} F_{pk + pq} G_{pk + pq + t} }}}  = F_{pn} \sum_{k = 1}^q {\frac{{G_{2pk + pn + t} }}{{F_{pk} G_{pk + t} F_{pk + pn} G_{pk + pn + t} }}}\,.
\]
\end{theorem}
\begin{proof}
Use $f(k)=1/(F_kG_{k+t})$ in Lemma~\ref{FS} with \mbox{$u=p$}, $v=q$ and $w=n$, while taking cognisance of the following identities which follow from identities~\eqref{equ.vtbejuo} and~\eqref{equ.p2n0qq9}:
\begin{equation}
\frac{1}{{F_{pk} G_{pk + t} }} - \frac{1}{{F_{pk + pq} G_{pk + pq + t} }} = \frac{{F_{pq} G_{2pk + pq + t} }}{{F_{pk} G_{pk + t} F_{pk + pq} G_{pk + pq + t} }},\quad\mbox{$pq$ even}\,,
\end{equation}
and
\begin{equation}
\frac{1}{{F_{pk} G_{pk + t} }} - \frac{1}{{F_{pk + pn} G_{pk + pn + t} }} = \frac{{F_{pn} G_{2pk + pn + t} }}{{F_{pk} G_{pk + t} F_{pk + pn} G_{pk + pn + t} }},\quad\mbox{$pn$ even}\,.
\end{equation}

\end{proof}
\begin{theorem}
If $p$, $q$, $n$ and $t$ are positive integers such that $p$ is odd or $n$ and $q$ are even, then
\[
F_{pq} \sum_{k = 1}^n {\frac{(- 1)^{k-1}{G_{2pk + pq + t} }}{{F_{pk} G_{pk + t} F_{pk + pq} G_{pk + pq + t} }}}  = F_{pn} \sum_{k = 1}^q {\frac{(- 1)^{k-1}{G_{2pk + pn + t} }}{{F_{pk} G_{pk + t} F_{pk + pn} G_{pk + pn + t} }}}\,.
\]
\end{theorem}
\subsection{Horadam sequence}\label{sec.horadam}
Some of the above results can be extended to the Horadam sequence~\cite{horadam}, $\{W_i\} = \{W_i(a,b; P,Q)\}$ defined by
\begin{equation}
W_0=a,W_1 = b,W_i = PW_{i-1}-QW_{i-2}\,,  (i>2)\,,
\end{equation}
where $a$, $b$, $P$, and $Q$ are integers, with $PQ\ne 0$ and $\Delta=P^2-4Q>0$. We define the sequence $\{U_i\}$ by $U_i =W_i(0,1; P,Q)$  and note also that our sequence $\{G_i\}$ is given by $G_i =W_i(G_0,G_1; 1,-1)$. It is readily established that~\cite{horadam,jeannin}:
\begin{equation}
W_i=\frac{a\alpha^i-b\beta^i}{\alpha-\beta}\,,
\end{equation}
where $\alpha=(P+\sqrt\Delta)/2$, $\beta=(P-\sqrt\Delta)/2$, $A=b-\beta a$ and $B=b-\alpha a$.
\begin{theorem}\label{thm.m3wxctm}
If $n$ and $q$ are nonnegative integers and $p$ is a nonzero integer, then
\[
U_{pq} \sum_{k = 1}^n {\frac{{Q^{pk} }}{{W_{pk} W_{pk + pq} }}}  = U_{pn} \sum_{k = 1}^q {\frac{{Q^{pk} }}{{W_{pk} W_{pk + pn} }}}\,.
\]
\end{theorem}
Note that when $p=1$, Theorem~\ref{thm.m3wxctm} reduces to Theorem~1 of~\cite{jeannin}.
\begin{proof}
Since $n$ and $k$ in identity~(4.1) of~\cite{jeannin} are arbitrary nonnegative integers, we substitute $pk$ for $n$ and $pq$ for $k$ in the identity, obtaining
\begin{equation}\label{equ.u3hkcuo}
\frac{{\beta ^{pk} }}{{W_{pk} }} - \frac{{\beta ^{pk + pq} }}{{W_{pk + pq} }} = \frac{{AQ^{pk} U_{pq} }}{{W_{pk} W_{pk + pq} }}\,.
\end{equation}
The theorem now follows by choosing $f(k)=\beta^k/W_k$ in Lemma~\ref{FS} with $w=n$, $u=p$ and $v=q$ while making use of~\eqref{equ.u3hkcuo}.
\end{proof}
\begin{theorem}\label{thm.iplip0k}
If $n$ and $q$ are nonnegative \underline{even} integers and $p$ is a nonzero integer, then
\[
U_{pq} \sum_{k = 1}^{n} {\frac{{(\pm Q^p)^{k} }}{{W_{pk} W_{pk + pq} }}}  = U_{pn} \sum_{k = 1}^{q} {\frac{{(\pm Q^p)^{k} }}{{W_{pk} W_{pk + pn} }}}\,.
\]
\end{theorem}
\begin{proof}
The theorem follows by choosing $f(k)=\beta^k/W_k$ in Lemma~\ref{FS1} with $w=n$, $u=p$ and $v=q$, while making use of~\eqref{equ.u3hkcuo}.
\end{proof}

\end{document}